\documentclass[11pt, oneside]{article}
\usepackage{amsfonts}
\usepackage{amsfonts}
\usepackage{mathrsfs}
\usepackage{amsfonts}

\usepackage{latexsym,amssymb,amsmath}
\usepackage{latexsym}
\usepackage{amssymb}
\usepackage{amsmath}
\usepackage{amsthm}
\usepackage{indentfirst}
\usepackage{color}
\usepackage{graphicx}
\numberwithin{equation}{section} \allowdisplaybreaks

\newtheorem{theorem}{\color{black}\indent Theorem}[section]
\newtheorem{lemma}{\color{black}\indent Lemma}[section]

\newtheorem{remark}{\color{black}\indent Remark}[section]

\usepackage{amssymb,amsmath}
\begin{document}
\title{\LARGE\bf Study of Solutions for a quasilinear
Elliptic Problem With negative exponents
\\
\author{ Bin Guo\thanks{Corresponding author\newline \hspace*{6mm}{\it Email
addresses:}~bguo1029@gmail.com(Bin Guo),~wjgao@jlu.edu.cn (Wenjie
Gao).},~ Wenjie Gao, Yanchao Gao
\\
\small{\it{School of Mathematics, Jilin University,} \it{Changchun
130012, PR China}}} }
\date{} \maketitle
{\bf Abstract:} The authors of this paper deal with the existence
 and regularities of weak solutions to
 the homogenous $\hbox{Dirichlet}$ boundary value problem for the equation
 $-\hbox{div}(|\nabla u|^{p-2}\nabla
 u)+|u|^{p-2}u=\frac{f(x)}{u^{\alpha}}$.  The authors apply the method of
regularization and $\hbox{Leray-Schauder}$ fixed point theorem as well as
a necessary compactness argument to prove the existence of solutions and then obtain some maximum norm
estimates by constructing three suitable iterative sequences. Furthermore, we find
that the critical exponent of $m$ in $\|f\|_{L^{m}(\Omega)}$. That is, when $m$ lies in different intervals, the solutions of the problem mentioned
belongs to different $\hbox{Sobolev}$ spaces. Besides,
we prove that the solution of this problem is not in
$W^{1,p}_{0}(\Omega)$ when $\alpha>2$, while the solution of this problem is in
$W^{1,p}_{0}(\Omega)$ when $1<\alpha<2$.


\section{Introduction}
In this paper, we study existence and regularity of solutions for
the following quasi-linear elliptic problem
\begin{equation}
\begin{cases}
-\textrm{div}(|\nabla u|^{p-2}\nabla
u)+|u|^{p-2}u=\displaystyle\frac{f(x)}{u^{\alpha}},&x\in\Omega,\\
u=0,&x\in~\partial\Omega,
\end{cases}
\end{equation}
where $\Omega$ is a bounded open domain in $R^{N}(N\geqslant1)$ with
smooth boundary $\partial\Omega.$ $f\geqslant0,~~f\not\equiv0,~f\in
L^{1}(\Omega),~~p>1,~\alpha>0.$

Model $(1.1)$ may describe many physical phenomena such as
non-Newtonian flows in porous media, chemical heterogeneous
catalysts, nonlinear heat equations etc.
\cite{MOTANI,GDMFMLAP,LBTG,LBTG1}. When $p=2,$ many authors have
studied this problem. In 1991, $\hbox{Lazer}$ and $\hbox{Mckenna}$
\cite{ACLPJM} dealt with the case when $f$ was a continuous
function. They proved that the solution belonged to
$W^{1,2}_{0}(\Omega)$ if and only if $\alpha<3,$ while it was not in
$C^{1}(\overline{\Omega})$ if $\alpha>1.$ Later,
$\hbox{Lair,~Shaker, Zhang,~and~Cheng}$ generalized this results
\cite{ZZJC,YSSWYL,ALAS}. Moreover, $\hbox{Boccardo~and~Orsina}$ in
\cite{LBLO} discussed how the summability of $f$ and the values of
$\alpha$ affected the existence, regularity and nonexistence of
solutions. For more results, the interested readers may refer to
\cite{NHCSNS,MCGP}. In the case when $p\neq2,$
$\hbox{Giacomoni,~Schindler~and~Tak\'{a}\v{c}}$ in \cite{JGJST}
applied upper-lower solution method and a mountain pass theorem
to prove that this problem had multiple weak solutions. And then,
the authors in \cite{NHLKS} not only improved the results in
\cite{JGJST} but also obtained that the solution was not in
$W^{1,p}_{0}(\Omega)$ if $\alpha>\frac{2p-1}{p-1}.$  For more
properties of solutions, we may refer to \cite{PT,FCMGVR}. We point
out that the first eigenfunction of the $p-\hbox{Laplacian}$ operator with
homogenous boundary value problem plays a major role in
all of the papers mentioned above. Besides,
$\hbox{Boccardo~and~Orsina}$ claimed that the results such as Lemma
3.3,4.3 and 5.5 of \cite{LBLO} may be generalized to the case when
the linear differential operator was replaced by a monotone
differential operator, for example, $p-\hbox{Laplacian}$ operator.
We find that the singular problem involving $p-\hbox{Laplacian}$
operator is more complicated. Especially, whether or not the problem
has a solution in $W^{1,p}_{0}(\Omega)$ depends on the relations of
$\alpha,p,N$ as well as the summability of $f.$  In this paper, we
discuss separately the properties of solutions to the problem when
$\alpha=1,>1$ and $0<\alpha<1$. First, we apply the method of
regularization and $\hbox{Leray-Schauder}$ fixed point theorem as well as
a necessary compactness argument to overcome some difficulties
arising from the nonlinearity of the differential operator and the
singularity of the nonlinear terms and then obtain existence of solutions. Moveover, some maximum norm
estimates are obtained by constructing three suitable iterative sequences.
Secondly, we consider the case $0<\alpha<1.$  By means of the maximum principle,
$\hbox{Hopf}$ Lemma and a partition of unity argument, we prove that the solution of the  following problem
\begin{equation}
\begin{cases}
-\textrm{div}(|\nabla u_{1}|^{p-2}\nabla
u_{1})+|u_{1}|^{p-2}u_{1}=\displaystyle\frac{\min\{f(x),1\}}{(|u_{1}|+1)^{\alpha}},&x\in\Omega,\\
u_{1}=0,&x\in~\partial\Omega,
\end{cases}
\end{equation}
satisfies $\displaystyle\int_{\Omega}|u_{1}|^{-r}dx<\infty,~\forall~r<1.$ And then applying this result, we find an interesting phenomenon:
\begin{align*}
&(1)\hbox{If}~f(x)\in L^{m}(\Omega) ~for~m>m^{*}=\frac{Np}{Np-(1-\alpha)(N-p)}>1,~\hbox{then Problem (1.1) has a }\\
&~~~~~\hbox{solution in~} W^{1,p}_{0}(\Omega);\\
&(2)\hbox{If}~f(x)\in L^{m}(\Omega) ~for~1<m<m^{*},~\hbox{then Problem (1.1) has a solution in~} W^{1,q}_{0}(\Omega)~q<p.
\end{align*}
In other words, the value of $m$ in $L^{m}(\Omega)$ norm of $f(x)$
determines
whether or not Problem $(1.1)$ has a solution in
$W^{1,p}_{0}(\Omega).$
In the case of $\alpha>1,$
we prove that this problem does not have a solution in
$W^{1,p}_{0}(\Omega)$ when $\alpha>2$, where the function $f(x)$ is
permitted not to be strictly positive on $\overline{\Omega}$, our
result is more general than that of \cite{NHLKS}. Furthermore, we apply the result $\int_{\Omega}|u_{1}|^{-r}dx<\infty,~\forall~r<1$ to prove that
Problem $(1.1)$ has a unique solution in $W^{1,p}_{0}(\Omega)$ in the case when $1<\alpha<2.$

\section{The case when $\alpha=1$}
In this section, we apply the method of regularization and
$\hbox{Leray-Schauder}$ fixed point theorem to prove the existence of
solutions. Besides, we obtain $L^{\infty}$ norm estimates by Morse
iteration technique.

In order to prove the main results of this section, we consider the
following auxiliary problem
\begin{equation}
\begin{cases}
-\textrm{div}(|\nabla u_{n}|^{p-2}\nabla
u_{n})+|u_{n}|^{p-2}u_{n}=\displaystyle\frac{f_{n}(x)}{(|u_{n}|+\frac{1}{n})^{\alpha}},&x\in\Omega,\\
u_{n}=0,&x\in~\partial\Omega,
\end{cases}
\end{equation}
where $f_{n}=\min\{f(x),n\}.$

\begin{lemma}
Problem $(2.1)$ has a unique nonnegative solution $u_{n}\in
W^{1,p}_{0}(\Omega)\cap L^{\infty}(\Omega)$ for any fixed $n\in
N^{*}.$
\end{lemma}
\begin{proof}
Let $n\in N$ be fixed.  For any $w\in L^{p}(\Omega),~\sigma\in[0,1]$, we get that
the following problem has a unique solution $v\in
W^{1,p}_{0}(\Omega)\cap L^{\infty}(\Omega)$ by applying a
variational method
\begin{equation}
\begin{cases}
-\textrm{div}(|\nabla v|^{p-2}\nabla
v)+|v|^{p-2}v=\displaystyle\frac{\sigma f_{n}(x)}{(|w|+\frac{1}{n})^{\alpha}},&x\in\Omega,\\
v=0,&x\in~\partial\Omega.
\end{cases}
\end{equation}
So, for any $w\in L^{p}(\Omega),~\sigma\in[0,1]$, we may define the map
$\Gamma:~[0,1]\times L^{p}(\Omega)\rightarrow~L^{p}(\Omega)~\hbox{as}~\Gamma(\sigma,w)=v.$
It is easy to see that $\Gamma(0,w)=0.$ For all $v\in L^{p}(\Omega)$ satisfying $\Gamma(\sigma,v)=v,$ we have $\|v\|_{L^{p}(\Omega)}\leqslant C,$
where the positive constant $C$ depends on $n.$
In fact, multiplying the first identity in $(2.2)$ by $v$, and integrating
over $\Omega,$ we have
\begin{align*}
\int_{\Omega}|\nabla
v|^{p}dx+\int_{\Omega}|v|^{p}dx=\int_{\Omega}\frac{\sigma f_{n}}{(|w|+\frac{1}{n})^{\alpha}}v
dx\leqslant n^{\alpha+1}\int_{\Omega}|v|dx.
\end{align*}
Applying the embedding theorem $W^{1,p}(\Omega)\hookrightarrow
L^{1}(\Omega)$, we obtain
\begin{align*}
\|v\|^{p}_{W^{1,p}}\leqslant Cn^{\alpha+1}\|v\|_{W^{1,p}},
\end{align*}
which implies
$$\|v\|_{W^{1,p}}\leqslant Cn^{\frac{\alpha+1}{p-1}}.$$
Due to the embedding $W^{1,p}(\Omega)\overset{compact}\hookrightarrow L^{p}(\Omega),$ we get that the map $\Gamma$ is a compact operator and  $\|v\|_{L^{p}(\Omega)}\leqslant C(n).$ Then
by $\hbox{Leray-Schauder's}$ fixed point theorem, we know that there exists a
$u_{n}\in W^{1,p}_{0}(\Omega)$ such that $u_{n}=\Gamma(1,u_{n})$,
i.e.~Problem $(2.1)$ has a solution. Noting that
$\frac{f_{n}}{(|u_{n}|+\frac{1}{n})^{\alpha}}\geqslant0,$ the
maximum principle in \cite{LOAUNN,ADRD} shows that $u_{n}\geqslant0,~u_{n}\in L^{\infty}(\Omega).$
\end{proof}
\begin{lemma}
The sequence $u_{n}$ is increasing with respect to $n$. $u_{n}>0$ in
$\Omega'$ for any $\Omega'\subset\subset\Omega,$ and there exists a
positive constant $C_{\Omega'}$ (independent of $n$) such that for
all $n\in N^{*}$
\begin{align}
u_{n}\geqslant C_{\Omega'}>0,~for~every~x\in\Omega'.
\end{align}
\end{lemma}
\begin{proof}
Choosing $(u_{n}- u_{n+1})_{+}=\max\{u_{n}-u_{n+1},0\}$ as a test
function, observing that
\begin{align*}
&(|\nabla u_{n}|^{p-2}\nabla u_{n}-|\nabla u_{n+1}|^{p-2}\nabla u_{n+1})(\nabla u_{n}-\nabla u_{n+1})_{+}\geqslant0,\\
&[(u_{n+1}+\frac{1}{n+1})^{\alpha}-(u_{n}+\frac{1}{n})^{\alpha}](u_{n}-u_{n+1})_{+}\leqslant0,~~for~every~~\alpha>0,\\
&0\leqslant f_{n}\leqslant f_{n+1},
\end{align*}
we get
$$0\leqslant\int_{\Omega}|\nabla (u_{n}-u_{n+1})_{+}|^{p}dx\leqslant0.$$
This inequality yields $(u_{n}-u_{n+1})_{+}=0$~a.e.~in~$\Omega$,
that is $u_{n}\leqslant u_{n+1}$ for every $n\in N^{*}.$ Since the
sequence $u_{n}$ is increasing with respect to $n$, we only need to
prove that $u_{1}$ satisfies Inequality $(2.3).$ According to Lemma
2.1, we know that there exists a positive constant $C$ (depending
only on $|\Omega|,N,p$) such that
$\|u_{1}\|_{L^{\infty}(\Omega)}\leqslant
C\|f_{1}\|_{L^{\infty}(\Omega)}\leqslant C,$ then
\begin{align*}
-\textrm{div}(|\nabla u_{1}|^{p-2}\nabla
u_{1})+|u_{1}|^{p-2}u_{1}=\frac{f_{1}}{(u_{1}+1)^{\alpha}}
\geqslant\frac{f_{1}}{(C+1)^{\alpha}}.
\end{align*}
Noting that $\frac{f_{1}}{(C+1)^{\alpha}}\geqslant0,~
\frac{f_{1}}{(C+1)^{\alpha}}\not\equiv0,$ the strong maximum
principle implies that $u_{1}>0~\text{in}~\Omega,$ i.e. Inequality
$(2.3)$ holds.
\end{proof}
Our main results are the following
\begin{theorem}
Let $f$ be a nonnegative function in $L^{1}(\Omega)~(f\not\equiv0)$.
Then Problem $(1.1)$ has a solution $u\in W^{1,p}_{0}(\Omega)$
satisfying
\begin{equation*}
\int_{\Omega}|\nabla u|^{p-2}\nabla u\nabla\varphi
dx+\int_{\Omega}|u|^{p-2}u\varphi
dx=\int_{\Omega}\frac{f\varphi}{u}dx,~~\forall \varphi\in
C^{\infty}_{0}(\Omega).
\end{equation*}
Moreover, suppose that $f\in L^{m}(m\geqslant1),$ then the solution
$u$ of Problem $(1.1)$ satisfies the following properties
\begin{align*}
&~(i)~\text{If} ~~m>\frac{N}{p}~~and~~2-\frac{2}{m}<p<N,
~then ~u\in L^{\infty}(\Omega). \\
&~(ii)~\text{If}~~1\leqslant m<\frac{N}{p}, ~then ~u\in
L^{s}(\Omega), for~1<s\leqslant\frac{pNm}{N-pm}.
\end{align*}
\end{theorem}
\begin{proof}
$\mathbf{Part 1}$ (Existence).
Multiplying the first identity in Problem $(2.1)$ by $u_{n}$ and
integrating over $\Omega,$ we get
\begin{equation*}
\int_{\Omega}|\nabla
u_{n}|^{p}dx+\int_{\Omega}|u_{n}|^{p}dx=\int_{\Omega}\frac{f_{n}u_{n}}{u_{n}+\frac{1}{n}}dx\leqslant\int_{\Omega}|f_{n}|dx
\leqslant\int_{\Omega}|f|dx,
\end{equation*}
i.e. $\|u_{n}\|_{W^{1,p}_{0}(\Omega)}\leqslant
\|f\|^{\frac{1}{p}}_{L^{1}(\Omega)}.$

Then we know that there exist $u\in W^{1,p}(\Omega)$ and
$\vec{V}\in L^{\frac{p}{p-1}}(\Omega,~R^{N})$ such that
\begin{equation}
\begin{cases}
\begin{split}
&u_{n}\rightharpoonup u ~~weakly~in~~
W^{1,p}(\Omega)~and~strongly~in~L^{p}(\Omega);\\
&u_{n}\rightarrow u~~~~ a,e.~in~~~\Omega;\\
&|\nabla u_{n}|^{p-2}\nabla u_{n}\rightharpoonup \vec{V}
~~~weakly~in~~~ L^{\frac{p}{p-1}}(\Omega,~R^{N}).
\end{split}
\end{cases}
\end{equation}
For every $\varphi\in C^{\infty}_{0}(\Omega),$ we get from Inequality
$(2.3)$ that
\begin{align*}
0\leqslant\Big|\frac{f_{n}\varphi}{u_{n}+\frac{1}{n}}\Big|\leqslant\frac{\|\varphi\|_{L^{\infty}}}{C_{\Omega'}}f(x),
\end{align*}
where $\Omega'=\{x:\varphi\neq0\}.$ Then applying \hbox{Lebesgue}
Dominated Convergence Theorem, one has that
\begin{equation}
\lim\limits_{n\rightarrow+\infty}\int_{\Omega}\frac{f_{n}\varphi}{u_{n}+\frac{1}{n}}dx=\int_{\Omega}\frac{f\varphi}{u}dx.
\end{equation}
Since $u_{n}$ satisfies the following identity
\begin{equation}
\int_{\Omega}|\nabla u_{n}|^{p-2}\nabla u_{n}\nabla\varphi
dx+\int_{\Omega}|u_{n}|^{p-2}u_{n}\varphi
dx=\int_{\Omega}\frac{f_{n}\varphi}{u_{n}+\frac{1}{n}}dx,~~\forall
\varphi\in C^{\infty}_{0}(\Omega).
\end{equation}
Combining $(2.4)-(2.6)$, we have
\begin{equation}
\int_{\Omega}\vec{V}\nabla\varphi dx+\int_{\Omega}|u|^{p-2}u\varphi
dx=\int_{\Omega}\frac{f\varphi}{u}dx,~~\forall \varphi\in
C^{\infty}_{0}(\Omega).
\end{equation}
Next, we shall prove $\vec{V}=|\nabla u|^{p-2}\nabla
u,~~a.e.~in~\Omega.$ Since $C^{\infty}_{0}(\Omega)$ is dense in $W^{1,p}_{0}(\Omega)$, we may choose  $u_{n}-\xi$ with $\xi\in
C^{\infty}_{0}(\Omega)$ as a test function to get
\begin{equation}
\int_{\Omega}|\nabla u_{n}|^{p-2}\nabla u_{n}\nabla(u_{n}-\xi)
dx+\int_{\Omega}|u_{n}|^{p-2}u_{n}(u_{n}-\xi)
dx=\int_{\Omega}\frac{f_{n}(u_{n}-\xi)}{u_{n}+\frac{1}{n}}dx.
\end{equation}
Noting that $(|\nabla u_{n}|^{p-2}\nabla u_{n}-|\nabla
\xi|^{p-2}\nabla \xi)(\nabla u_{n}-\nabla\xi)\geqslant0,$ we obtain
that
\begin{equation}
\int_{\Omega}|\nabla \xi|^{p-2}\nabla \xi\nabla(u_{n}-\xi)
dx+\int_{\Omega}|u_{n}|^{p-2}u_{n}(u_{n}-\xi)
dx\leqslant\int_{\Omega}\frac{f_{n}(u_{n}-\xi)}{u_{n}+\frac{1}{n}}dx.
\end{equation}
Letting $ n\rightarrow\infty$ in $(2.9)$ and using Identity $(2.7)$,
one get
\begin{equation}
\int_{\Omega}(|\nabla \xi|^{p-2}\nabla\xi-\vec{V})\nabla(u-\xi)
dx\leqslant0,~~\forall~\xi\in C^{\infty}_{0}(\Omega).
\end{equation}
Choosing $\xi=u\pm\varepsilon\varphi~\hbox{with}~\varphi\in C^{\infty}_{0}(\Omega)$ and letting $\varepsilon\rightarrow0^{+}$, we have
$\vec{V}=|\nabla u|^{p-2}\nabla u,~a.e.~in~\Omega.$
This proves that $u$ is a weak solution of Problem $(1.1)$.

$\mathbf{Part 2}$ (Regularity). (i) choosing
$u_{n}^{\frac{\beta_{k}}{m'}}u_{n}(\beta_{k}\geqslant m')$ as a test
function in $(2.6)$, we obtain
\begin{equation}
\begin{split}
(\frac{\beta_{k}}{m'}+1)\int_{\Omega}|\nabla
u_{n}|^{p}u_{n}^{\frac{\beta_{k}}{m'}}dx
&+\int_{\Omega}u_{n}^{\frac{\beta_{k}}{m'}+p}dx=
\int_{\Omega}\frac{f_{n}(x)}{u_{n}+\frac{1}{n}}u_{n}^{\frac{\beta_{k}}{m'}+1}dx\\
&\leqslant\int_{\Omega}f(x)u_{n}^{\frac{\beta_{k}}{m'}}
\leqslant\|f\|_{L^{m}(\Omega)}\|u_{n}\|^{\frac{\beta_{k}}{m'}}_{L^{\beta_{k}}}.
\end{split}
\end{equation}
Define two sequences
\begin{align*}
\beta_{k+1}=\frac{\beta_{k}^{*}p^{*}}{pm'},~~\beta_{k}^{*}=\beta_{k}+pm'
,\beta_{1}=pm',~p^{*}=\frac{Np}{N-p},~m'=\frac{m}{m-1}.
\end{align*}
According to the condition
$m>\frac{N}{p},~2-\frac{2}{m}<p\Rightarrow 2<\beta_{1}<p^{*}$ and
applying $\textrm{Sobolev}$ embedding theorem
$W^{1,p}_{0}(\Omega)\hookrightarrow L^{q}(\Omega)$ with $1\leqslant
q\leqslant p^{*}$, we get
\begin{equation}
\|u_{n}\|_{L^{\beta_{1}}(\Omega)}^{\beta_{1}}=\|u_{n}\|_{L^{pm'}(\Omega)}^{pm'}\leqslant
C\|u_{n}\|_{W^{1,p}_{0}(\Omega)}\leqslant
C\|f\|^{\frac{1}{p}}_{L^{1}(\Omega)}.
\end{equation}
Once again, applying $\textrm{Sobolev}$ embedding theorem
$W^{1,p}_{0}(\Omega)\hookrightarrow L^{p^{*}}(\Omega)$, we get
\begin{equation}
\mu^{-p}\Big(\frac{pm'}{\beta_{k}^{*}}\Big)^{p}\|u_{n}^{\frac{\beta^{*}_{k}}{pm'}}\|^{p}_{L^{p^{*}}(\Omega)}
\leqslant\Big(\frac{pm'}{\beta_{k}^{*}}\Big)^{p}\|\nabla
u^{\frac{\beta^{*}_{k}}{pm'}}_{n}\|^{p}_{L^{p}(\Omega)}\leqslant\int_{\Omega}|\nabla
u_{n}|^{p}u_{n}^{\frac{\beta_{k}}{m'}}dx.
\end{equation}
Combining $(2.13)$ with $(2.11)$ and using the definition of
$\beta_{k}$, we have
\begin{equation}
\|u_{n}\|_{\beta_{k+1}}^{\frac{\beta^{*}_{k}}{m'}}\leqslant\mu^{p}
\Big(\frac{\beta_{k}^{*}}{pm'}\Big)^{p}\|f\|_{m}\|u_{n}\|^{\frac{\beta_{k}}{m'}}_{\beta_{k}},
\end{equation}
with $\|u_{n}\|_{\sigma}=\|u_{n}\|_{L^{\sigma}(\Omega)}.$ In order
to prove that $\|u_{n}\|_{\infty}$ is bounded with a bound
independent of $n$, we use a trick in \cite{MOTANI}. Let
$F_{k}=\beta_{k}\ln \|u_{n}\|_{\beta_{k}}.$ By means of Inequality
$(2.14)$, we get
\begin{equation}
\begin{split}
F_{k+1}&=\beta_{k+1}\ln \|u_{n}\|_{\beta_{k+1}}\\
&\leqslant\frac{\beta_{k+1}m'}{\beta^{*}_{k}}\Big[p\ln\mu+p\ln\frac{\beta_{k}^{*}}{pm'}
+\ln\|f\|_{m}+\frac{\beta_{k}}{m'}\ln\|u_{n}\|_{\beta_{k}}\Big]\\
&\leqslant p^{*}[\ln\mu+\ln\frac{\beta_{k}^{*}}{pm'}]+
\frac{p^{*}}{p}\ln\|f\|_{m}+\frac{p^{*}}{pm'}F_{k}\\
&\leqslant \lambda_{k}+\delta F_{k},
\end{split}
\end{equation}
with
$\lambda_{k}=p^{*}\ln(\mu\beta^{*}_{k}\|f\|_{m}^{\frac{1}{p}}),~\delta=\frac{p^{*}}{pm'}>1.$
Since
$\beta_{k}=\frac{(N+p)m-2N}{(m-1)(mp-N)}\delta^{k-1}-\frac{Npm}{mp-N}$,
we obtain that
\begin{align*}
\lambda_{k}&=p^{*}\ln\Big[\mu\|f\|^{\frac{1}{p}}_{m}
\Big(\frac{(N+p)m-2N}{(m-1)(mp-N)}\delta^{k-1}-\frac{Npm}{mp-N}\Big)\Big]\\
&\leqslant p^{*}\ln[\mu\|f\|^{\frac{1}{p}}_{m}]+p^{*}\ln\delta^{k-1}
+p^{*}\ln\Big[\frac{(N+p)m-2N}{(m-1)(mp-N)}\Big]\\
&\leqslant b+p^{*}(k-1)\ln \delta,
\end{align*}
where
$b=p^{*}\ln\Big[\mu\|f\|^{\frac{1}{p}}_{m}\frac{(N+p)m-2N}{(m-1)(mp-N)}\Big].$
Using Inequalities $(2.15)$ and
$\beta_{k}=\frac{(N+p)m-2N}{(m-1)(mp-N)}\delta^{k-1}-\frac{Npm}{mp-N}$,
we have
\begin{align*}
F_{k}&\leqslant
\delta^{k-1}F_{1}+\lambda_{k-1}+\delta\lambda_{k-2}+\cdots+\delta^{k-2}\lambda_{1}\\
&\leqslant
\delta^{k-1}\ln(C\|f\|_{1}^{\frac{1}{p}})+\frac{(b+p^{*}\ln\delta)
\delta}{(\delta-1)^{2}}(\delta^{k-1}-1)+\frac{b}{1-\delta};\\
\frac{F_{k}}{\beta_{k}}
&\leqslant\frac{\delta^{k-1}\ln(C\|f\|_{1}^{\frac{1}{p}})+
\frac{(b+p^{*}\ln\delta)\delta}{(\delta-1)^{2}}(\delta^{k-1}-1)+\frac{b}{1-\delta}}
{\frac{(N+p)m-2N}{(m-1)(mp-N)}\delta^{k-1}-\frac{Npm}{mp-N}}\\
&~~~~~~~~\rightarrow\frac{(m-1)(mp-N)\Big[(\delta-1)^{2}\ln(C\|f\|_{1}^{\frac{1}{p}})+\delta
p^{*}\ln\delta+\delta b\Big]}{(\delta-1)^{2}(Nm+pm-2N)}:=
d_{0}~~as~~k\rightarrow+\infty.
\end{align*}
So
\begin{align*}
\|u_{n}\|_{\infty}\leqslant\lim\limits_{k\rightarrow\infty}\sup\|u_{n}\|_{\beta_{k}}\leqslant
\lim\limits_{k\rightarrow\infty}\sup\exp(\frac{F_{k}}{\beta_{k}})=e^{d_{0}}.
\end{align*}
Applying $\textrm{Fatou's}$~lemma, we
get$$\|u\|_{\infty}\leqslant\lim\limits_{n\rightarrow\infty}\inf\|u_{n}\|_{\infty}\leqslant
e^{d_{0}}.$$

 (ii).~Define $\delta=\frac{m(N-p)}{N-m p}.$ According to the
condition $1<m<\frac{N}{p}$, we get $\delta>1$. Choosing
$u_{n}^{p(\delta-1)+1}$ as a test function, and applying
$\hbox{H\"{o}lder's}$ inequality,  we arrive at
\begin{equation}
\begin{split}
(p(\delta-1)+1)\int_{\Omega}|\nabla u_{n}|^{p}u_{n}^{p(\delta-1)}dx+
\int_{\Omega}u_{n}^{p\delta}dx&=\int_{\Omega}\frac{f_{n}}{u_{n}+\frac{1}{n}}u_{n}^{p(\delta-1)+1}dx
dx\\&\leqslant\int_{\Omega}f(x)u_{n}^{p(\delta-1)}dx\\
&\leqslant\Big(\int_{\Omega}u_{n}^{p(\delta-1)\frac{m}{m-1}}dx\Big)^{\frac{m-1}{m}}\|f\|_{L^{m}(\Omega)}.
\end{split}
\end{equation}
Moreover, applying $\textrm{Sobolev}$ embedding theorem and using
Inequality $(2.16)$, we get
\begin{equation}
\begin{split}
\Big(\int_{\Omega}|u_{n}|^{\frac{\delta
Np}{N-p}}dx\Big)^{\frac{N-p}{N}}&\leqslant C\int_{\Omega}|\nabla
u_{n}^{\delta}|^{p}dx\leqslant
C\delta^{p}\int_{\Omega}u_{n}^{p(\delta-1)}|\nabla u_{n}|^{p}dx
\\&\leqslant\Big(\int_{\Omega}u_{n}^{p(\delta-1)\frac{m}{m-1}}dx\Big)^{\frac{r-1}{r}}\|f\|_{L^{m}(\Omega)}.
\end{split}
\end{equation}
Then for any $~1\leqslant s\leqslant\frac{N pm}{N-mp}$, we have
\begin{equation*}
\|u_{n}\|_{L^{s}(\Omega)}\leqslant
C\|f\|^{\frac{1}{p}}_{L^{m}(\Omega)}.
\end{equation*}
\end{proof}

\section{The case when $0<\alpha<1$}
In this section, consider the case $0<\alpha<1$. We find that
Problem $(1.1)$ has a solution in $W^{1,p}_{0}(\Omega)$ when $f\in
L^{m}(\Omega)$ for $1<m^{*}\leqslant m$ with
$m^{*}=\frac{Np}{Np-(1-\alpha)(N-p)}$, but it is not clear whether
Problem $(1.1)$ has a solution in $W^{1,p}_{0}(\Omega)$ when $f\in
L^{m}(\Omega)$ for $1<m<m^{*}$. Fortunately, we prove that Problem $(1.1)$
exists a solution in $W^{1,q}_{0}(\Omega)$ for $1<q<p$. Our main
results are the following
\begin{lemma}
Let $u_{n}$ be the solution of Problem $(2.1)$ and suppose that
$f\in L^{m}(\Omega)$, with $m\geqslant m^{*}.$ Then
$\|u_{n}\|_{W^{1,p}_{0}(\Omega)}\leqslant
\|f\|^{\frac{1}{p}}_{L^{m}(\Omega)}.$
\end{lemma}
\begin{proof}
Multiplying the first identity in Problem $(2.1)$ by $u_{n}$ and
integrating over $\Omega,$ we get
\begin{equation*}
\int_{\Omega}|\nabla
u_{n}|^{p}dx+\int_{\Omega}|u_{n}|^{p}dx=\int_{\Omega}\frac{f_{n}u_{n}}
{(u_{n}+\frac{1}{n})^{\alpha}}dx\leqslant\int_{\Omega}|f_{n}|u^{1-\alpha}_{n}dx
\leqslant\|f\|_{L^{m}(\Omega)}\|u^{(1-\alpha)}_{n}\|_{L^{m'}(\Omega)},
\end{equation*}
and note that
\begin{align*}
\Big(\int_{\Omega}|u_{n}|^{p*}dx\Big)^{\frac{p}{p^{*}}}\leqslant
C\int_{\Omega}|\nabla u_{n}|^{p}dx,
\end{align*}
i.e.~$\|u_{n}\|_{L^{p^{*}}(\Omega)}\leqslant
C\|f\|^{\frac{1}{p+\alpha-1}}_{L^{1}(\Omega)}.$ Furthermore, we get
$\|u_{n}\|_{W^{1,p}_{0}(\Omega)}\leqslant
C\|f\|^{\frac{1-\alpha}{p+\alpha-1}}_{L^{m}(\Omega)}.$
\end{proof}
Next, we give the first result of this section:
\begin{theorem}
Let $f$ be a nonnegative function in $L^{m}(\Omega)(f\not\equiv0)$
with $m\geqslant m^{*}>1.$  Then Problem $(1.1)$ has a solution
$u\in W^{1,p}_{0}(\Omega)$
satisfying
\begin{equation*}
\int_{\Omega}|\nabla u|^{p-2}\nabla u\nabla\varphi
dx+\int_{\Omega}|u|^{p-2}u\varphi
dx=\int_{\Omega}\frac{f\varphi}{u^{\alpha}}dx,~~\forall \varphi\in
C^{\infty}_{0}(\Omega).
\end{equation*}
Moreover, the solution $u$ of Problem $(1.1)$ satisfies the following properties
\begin{align*}
&~(i)~\hbox{If} ~~m>\frac{N}{p}~~and~~2-\frac{2}{m}<p<N, ~then ~u\in L^{\infty}(\Omega). \\
&~(ii)~\hbox{If}~~1<m^{*}\leqslant m<\frac{N}{p}, ~then ~u\in
L^{s}(\Omega), with~1<s\leqslant\frac{(p+\alpha-1)Nm}{N-pm}.
\end{align*}
\end{theorem}
\begin{proof} $\mathbf{Part 1}$ (Existence).
Following the lines of the proof of Theorem 1.1, we obtain the
existence of the solution of Problem $ (1,1).$

$\mathbf{Part 2}$ (Regularity). (i) Using a trick of Theorem 1.1, it
is easy to see that the first conclusion holds.

 (ii).~Define $\delta=\frac{m(p+\alpha-1)(N-p)}{p(N-m p)}.$ According to the
condition $1< m^{*}\leqslant m<\frac{N}{p}$, we get $\delta>1$.
Choosing $u_{n}^{p(\delta-1)+1}$ as a test function, and applying
$\hbox{H\"{o}lder's}$ inequality,  we arrive at the following
relations
\begin{equation}
\begin{split}
(p(\delta-1)+1)\int_{\Omega}|\nabla u_{n}|^{p}u_{n}^{p(\delta-1)}dx+
\int_{\Omega}u_{n}^{p\delta}dx&=\int_{\Omega}\frac{f_{n}}{(u_{n}+\frac{1}{n})^{\alpha}}u_{n}^{p(\delta-1)+1}dx
dx\\&\leqslant\int_{\Omega}f(x)u_{n}^{p(\delta-1)+1-\alpha}dx\\
&\leqslant\Big(\int_{\Omega}u_{n}^{(p(\delta-1)+1-\alpha)\frac{m}{m-1}}dx\Big)^{\frac{m-1}{m}}\|f\|_{L^{m}(\Omega)}.
\end{split}
\end{equation}
Moreover, applying $\textrm{Sobolev}$ embedding theorem and using
Inequality $(2.11)$, we get
\begin{equation}
\begin{split}
\Big(\int_{\Omega}|u_{n}|^{\frac{\delta
Np}{N-p}}dx\Big)^{\frac{N-p}{N}}&\leqslant C\int_{\Omega}|\nabla
u_{n}^{\delta}|^{p}dx\leqslant
C\delta^{p}\int_{\Omega}u_{n}^{p(\delta-1)}|\nabla u_{n}|^{p}dx
\\&\leqslant\Big(\int_{\Omega}u_{n}^{p(\delta-1+1-\alpha)\frac{m}{m-1}}dx\Big)^{\frac{m-1}{m}}\|f\|_{L^{m}(\Omega)}.
\end{split}
\end{equation}

Using the definition of $\delta$, we get $\delta
p^{*}=(p(\delta-1)+1-\alpha)\frac{m}{m-1},$ then for all
$~1\leqslant s\leqslant\frac{Nm(p+\alpha-1)}{N-mp}$, we have
\begin{equation*}
\|u_{n}\|_{L^{s}(\Omega)}\leqslant
C\|f\|^{\frac{1}{p+\alpha-1}}_{L^{m}(\Omega)}.
\end{equation*}
\end{proof}

If $1\leqslant m<m^{*},$ it is not clear whether there exists a
solution $u\in W^{1,p}_{0}(\Omega)$ of Problem $(1.1).$ But, we may
find that Problem $(1.1)$ has a solution $u\in W^{1,q}_{0}(\Omega)$
with $q\leqslant
q_{*}\overset{\Delta}=\frac{Nm(p+\alpha-1)}{N-m(1-\alpha)}.$ Since
$1\leqslant m<m^{*},$ we get that
\begin{align*}
q_{*}-p=\frac{m[Np-(1-\alpha)(N-p)]-Np}{N-m(1-\alpha))}=\frac{Np(m-m^{*})}{m^{*}(N-m(1-\alpha))}<0\Rightarrow
W^{1,p}(\Omega)\subset W^{1,q}(\Omega),
\end{align*}
which implies that Problem $(1.1)$ has a solution in a larger space
$W^{1,q_{*}}(\Omega)$ rather than $W^{1,p}(\Omega).$

In this case, our main result is
\begin{theorem}
Suppose that $f\in L^{m}(\Omega)$ with $1\leqslant m<m^{*}$, then,
for any $\frac{Np}{N+p}<q\leqslant q_{*},$ Problem $(1.1)$ has a
weak solution $u\in W^{1,q}(\Omega)$ satisfying
\begin{equation*}
\int_{\Omega}|\nabla u|^{p-2}\nabla u\nabla\varphi
dx+\int_{\Omega}|u|^{p-2}u\varphi
dx=\int_{\Omega}\frac{f\varphi}{u^{\alpha}}dx,~~\forall \varphi\in
C^{\infty}_{0}(\Omega),
\end{equation*}
if the following assumptions hold
\begin{align*}
&(H_{1})1<p<\sqrt{N},~0<\alpha<~\frac{N-p^{2}}{N(p+1)},~and~\frac{Np}{N(p+\alpha-1)+p^{2}}<m<m^{*}~or~\\
&~~~~~~~~\sqrt{N}\leqslant
p<N,~0<\alpha<~1-\frac{1}{m},~and~1<m<m^{*}.
\end{align*}
\end{theorem}
In order to prove Theorem 3.2, we need the following lemmas
\begin{lemma}
The solution $u_{n}$ of Problem $(2.1)$ satisfies the estimates
\begin{equation}
\|u_{n}\|_{W^{1,q}(\Omega)}\leqslant
C:=C(\|f\|_{L^{m}},p,N,\alpha,|\Omega|),~1<q\leqslant q_{*}.
\end{equation}
\end{lemma}

\begin{proof}  Noting that $1\leqslant
m<m^{*}\Rightarrow
\frac{p+\alpha-1}{p}<\delta=\frac{m(p+\alpha-1)(N-p)}{p(N-m p)}<1,$
we can't choose $u_{n}^{p(\delta-1)+1}$ as a test-function because
the gradient of the function $u_{n}^{p(\delta-1)+1}$ is singular at $u_{n}=0.$  In
order to overcome this difficulty, for any $0<\theta<\frac{1}{n}$,
we may replace $u_{n}^{p(\delta-1)+1}$ by
$(u_{n}+\theta)^{p(\theta-1)+1}-\theta^{p(\delta-1)+1}$ in $(2.1),$
and we get
\begin{equation}
\begin{split}
(p\delta-p+1)\int_{\Omega}|\nabla
u_{n}|^{p}(u_{n}+\theta)^{p(\delta-1)}dx&+
\int_{\Omega}u_{n}^{p-1}[(u_{n}+\theta)^{p(\delta-1)+1}-\theta^{p(\delta-1)+1}]dx\\&=
\int_{\Omega}\frac{f_{n}}{(u_{n}+\frac{1}{n})^{\alpha}}[(u_{n}+\theta)^{p(\delta-1)+1}-\theta^{p(\delta-1)+1}]dx.
\end{split}
\end{equation}
Then, applying $\textrm{H\"{o}lder}'s$ inequality, we have
\begin{align*}
\frac{(p\delta-p+1)}{\delta^{p}}\int_{\Omega}|\nabla[(u_{n}+\theta)^{\delta}-\theta^{\delta}]|^{p}dx&+
\int_{\Omega}u_{n}^{p-1}(u_{n}+\theta)^{p(\delta-1)+1}-\theta^{p(\delta-1)+1}dx\\
&=\int_{\Omega}\frac{f_{n}}{(u_{n}+\frac{1}{n})^{\alpha}}[(u_{n}+\theta)^{p(\delta-1)+1}-\theta^{p(\delta-1)+1}]dx\\
&\leqslant\int_{\Omega}f(x)(u_{n}+\theta)^{p(\delta-1)+1}dx\\
&\leqslant\|f\|_{L^{m}(\Omega)}\Big(\int_{\Omega}(u_{n}+\theta)^{[p(\delta-1)+1-\alpha]\frac{m}{m-1}}dx\Big)^{\frac{m-1}{m}}.
\end{align*}
 Combining the above inequalities with $\textrm{Sobolev}$ embedding theorem
$W^{1,p}_{0}(\Omega)\hookrightarrow L^{p^{*}}(\Omega)$, we have
\begin{align}
\Big(\int_{\Omega}[(u_{n}+\theta)^{\delta}-\theta^{\delta}]^{p^{*}}dx\Big)^{\frac{p}{p*}}\leqslant
\|f\|_{L^{m}(\Omega)}\Big(\int_{\Omega}(u_{n}+\theta)^{[p(\delta-1)+1-\alpha]\frac{m}{m-1}}dx\Big)^{\frac{m-1}{m}(\Omega)}.
\end{align}
Letting $\theta\rightarrow0^{+}$ in $(3.5)$ and using the definition
of $\delta$, we obtain the following relations
\begin{align*}
\int_{\Omega}|u_{n}|^{\delta p^{*}}dx\leqslant
\|f\|^{\frac{mp^{*}}{pm-p^{*}(m-1)}}_{L^{m}(\Omega)},
\end{align*}
which implies
\begin{align}
\int_{\Omega}|u_{n}+\theta|^{\delta p^{*}}dx\leqslant
C(\|f\|^{\frac{mp^{*}}{pm-p^{*}(m-1)}}_{L^{m}(\Omega)}+\theta^{p^{*}}|\Omega|).
\end{align}
$(3.4)$ and $(3.6)$ yield that
\begin{align*}
\int_{\Omega}|\nabla
u_{n}|^{p}(u_{n}+\theta)^{p(\delta-1)}dx\leqslant
C(\|f\|^{\frac{mp}{pm-p^{*}(m-1)}}_{L^{m}(\Omega)}+\|f\|_{L^{m}(\Omega)}\theta^{\frac{(m-1)p^{*}}{m}}|\Omega|^{\frac{m-1}{m}}).
\end{align*}
Then for $1<q\leqslant q_{*}$, we have
\begin{equation*}
\begin{split}
\int_{\Omega}|\nabla u_{n}|^{q}dx&=\int_{\Omega}|\nabla
u_{n}|^{q}(u_{n}+\theta)^{(\delta-1)q}(u_{n}+\theta)^{(1-\delta)q}dx\\
&\leqslant \Big(\int_{\Omega}|\nabla
u_{n}|^{p}(u_{n}+\theta)^{(\delta-1)p}dx\Big)^{\frac{q}{p}}
\Big(\int_{\Omega}(u_{n}+\theta)^{\frac{(1-\delta)pq}{p-q}}dx\Big)^{1-\frac{q}{p}}\\
&\leqslant C.
\end{split}
\end{equation*}
\end{proof}

The following lemma plays a role in proving that Problem $(1.1)$ has
a nonnegative solution $u\in W^{1,q}_{0}(\Omega).$
\begin{lemma}
 The solution $u_{1}$ to Problem $(2.1)$ with $n=1$ satisfies
\begin{equation}
\int_{\Omega}u^{-r}_{1}dx<\infty,~ \forall~r<1.
\end{equation}
\end{lemma}
\begin{proof}
By $\frac{\min\{f(x),1\}}{(u_{1}+1)^{\alpha}}\leqslant1,$ and Lemma
2.2 in \cite{ADRD}, we know that there exists a $0<\beta<1$ such
that $u_{1}\in C^{1,\beta}(\overline{\Omega})$ and
$\|u_{1}\|_{C^{1,\beta}(\overline{\Omega})}\leqslant C,$ which implies that the
gradient of $u_{1}$ exists everywhere, then $\hbox{Hopf}$ Lemma in
\cite{JLV} shows that $\frac{\partial u_{1}(x)}{\partial \nu}>0,~in
~\overline{\Omega},$ where $\nu$ is the outward unit normal vector
of $\partial\Omega$ at $x.$ Moreover, following the lines of the
proof of Lemma in \cite{ACLPJM}, we get
\begin{align*}
\int_{\Omega}u_{1}^{r}dx<\infty,~\hbox{if and only if}~r>-1.
\end{align*}
\end{proof}

\textbf{Proof of Theorem 3.2.} According to Lemma 3.2, we know that
there exists a $u\in W^{1,q}_{0}(\Omega)$ such that
\begin{align}
\begin{cases}
&\displaystyle u_{n}\rightharpoonup u~~weakly~in~W^{1,q}_{0}(\Omega),\\
&\displaystyle u_{n}\rightarrow u~~strongly~in~L^{p}(\Omega),\\
&\displaystyle u_{n}\rightarrow u~~a.e.~in~\Omega.
\end{cases}
\end{align}
But the above convergence does not permit to pass to limit in the
following identity
\begin{equation}
\int_{\Omega}|\nabla u_{n}|^{p-2}\nabla u_{n}\nabla\varphi
dx+\int_{\Omega}|u_{n}|^{p-2}u_{n}\varphi
dx=\int_{\Omega}\frac{f_{n}\varphi}{(u_{n}+\frac{1}{n})^{\alpha}}dx,~for~\varphi\in
C^{\infty}_{0}(\Omega).
\end{equation}
We need to prove that $\nabla u_{n}$ converges to $\nabla u$ a.e. in
$\Omega$ or $\{\nabla u_{n}\}$ is a Cauchy sequence in measure or in
$L^{1}(\Omega).$

Case 1: $\sqrt{N}\leqslant
p<N,~0<\alpha<~1-\frac{1}{m},~and~1<m<m^{*}.$ Let $u_{n}$ and
$u_{k}$ be the solution Problem $(2.1)$ with $f_{n}$ and $f_{k}$,
respectively. Define
\begin{align*}
E_{n,k,\varepsilon}=\{x\in\Omega:|u_{n}-u_{k}|\leqslant\varepsilon\},
~T_{\varepsilon}(\tau)=\max\{\min\{\tau,\varepsilon\},-\varepsilon\},
\end{align*} then, we choose
$\varphi(x)=T_{\varepsilon}(u_{n}-u_{k})$ as a test function in
$(3.9)$ to get
\begin{align*}
\int_{\Omega}(|\nabla u_{n}|^{p-2}\nabla u_{n}-|\nabla
u_{k}|^{p-2}\nabla u_{k})\nabla T_{\varepsilon}(u_{n}-u_{k})dx&+
\int_{\Omega}(|u_{n}|^{p-2}u_{n}-|u_{k}|^{p-2}u_{k})T_{\varepsilon}(u_{n}-u_{k})dx\\
&=\int_{\Omega}(\frac{f_{n}}{(u^{n}+\frac{1}{n})^{\alpha}}-\frac{f_{k}}{(u^{k}+\frac{1}{k})^{\alpha}})
T_{\varepsilon}(u_{n}-u_{k})dx.
\end{align*}
Furthermore, we obtain
\begin{equation}
\begin{split}
\int_{E_{n,k,\varepsilon}}|\nabla
T_{\varepsilon}(u_{n}-u_{k})|^{p}dx=&
-\int_{\Omega}(|u_{n}|^{p-2}u_{n}-|u_{k}|^{p-2}u_{k})T_{\varepsilon}(u_{n}-u_{k})dx\\
&+\int_{\Omega}(\frac{f_{n}}{(u^{n}+\frac{1}{n})^{\alpha}}-\frac{f_{k}}{(u^{k}+\frac{1}{k})^{\alpha}})
T_{\varepsilon}(u_{n}-u_{k})dx:=I_{1}+I_{2}.
\end{split}
\end{equation}
Noting that $u_{n}$ is increasing with respect to $n$ and applying
the $\textrm{Sobolev}$ embedding theorem
$W^{1,q}(\Omega)\hookrightarrow L^{p}(\Omega)$, Lemma 3.3 and the
condition $0<\alpha<1-\frac{1}{m}$, we obtain
\begin{align}
&|I_{1}|\leqslant\varepsilon\int_{\Omega}(|u_{n}|^{p-1}+|u_{n}|^{p-1})dx\leqslant
C\varepsilon;\\
&|I_{2}|\leqslant\varepsilon\int_{\Omega}(\frac{f(x)}{u^{\alpha}_{1}}+\frac{f(x)}{u^{\alpha}_{1}})dx
\leqslant2\varepsilon\|f\|_{L^{m}}\|u_{1}^{-\alpha}\|_{L^{\frac{m}{m-1}}}\leqslant
C\varepsilon.
\end{align}
By $(3.9-3.12)$, we get
\begin{equation}
\int_{E_{n,k,\varepsilon}}|\nabla
T_{\varepsilon}(u_{n}-u_{k})|^{p}dx \leqslant C\varepsilon.
\end{equation}
Since $u_{n}\rightarrow u~strongly ~in ~L^{p}(\Omega),$ we get
\begin{equation}
measure\{x\in\Omega:~|u_{n}-u_{k}|>\varepsilon\}<\varepsilon.
\end{equation}
Combining $(3.12)-(3.14)$ with Lemma 3.2, we get
\begin{equation}
\begin{split}
\int_{\Omega}|\nabla
(u_{n}-u_{k})|dx&=\int_{E_{n,k,\varepsilon}}|\nabla
(u_{n}-u_{k})|+\int_{\Omega /E_{n,k,\varepsilon}}|\nabla
(u_{n}-u_{k})|\\
&\leqslant C\varepsilon^{\frac{1}{p}}+C(measure\{x\in\Omega:~|u_{n}-u_{k}|>\varepsilon\})^{1-\frac{1}{q}}\\
&\leqslant C\varepsilon^{\frac{1}{p}}+\leqslant
C\varepsilon^{1-\frac{1}{q}},
\end{split}
\end{equation}
which implies that $\{\nabla u_{n}\}$ is a \textrm{Cauchy} sequence
in $L^{1}(\Omega).$

By $(3.8),(3.9)~\hbox{and}~(3.15)$, we know that Problem $(1.1)$ has
a nonnegative solution $u\in W^{1,q}_{0}(\Omega).$

Case 2: $1<p<\sqrt{N},~0<\alpha<~\frac{N-p^{2}}{N(p+1)},~and~\frac{Np}{N(p+\alpha-1)+p^{2}}<m<m^{*}.$ Similarly as Case 1, we may prove the conclusion holds when
 $p,\alpha,N,m$ satisfy the second case.
\section{The case when $(\alpha>1)$}
In this section, we discuss how the value of $\alpha>1$ and the summability of $f$ affect the existence and regularities of solutions. First, we study the case when $f(x)\in L^{1}(\Omega).$ Our main results are as follows
\begin{lemma}
Let $u_{n}$ be the solution of Problem $(2.1)$ and suppose that
$f\in L^{1}(\Omega)$. Then
$\|u^{\frac{p+\alpha-1}{p}}_{n}\|_{W^{1,p}_{0}(\Omega)}\leqslant
\|f\|^{\frac{1}{p}}_{L^{1}(\Omega)}$ and
$\|u_{n}\|_{W^{1,p}_{loc}(\Omega)}\leqslant C\|f\|_{L^{1}(\Omega)}.$
\end{lemma}
\begin{proof} Our proof is similar as that in \cite{LBLO}. We give a brief proof.
Multiplying the first identity in Problem $(2.1)$ by
$u^{\alpha}_{n}$ and integrating over $\Omega,$ we get
\begin{equation*}
\alpha\int_{\Omega}|\nabla
u_{n}|^{p}u_{n}^{\alpha-1}dx+\int_{\Omega}|u_{n}|^{p+\alpha-1}dx=\int_{\Omega}\frac{f_{n}u^{\alpha}_{n}}
{(u_{n}+\frac{1}{n})^{\alpha}}dx\leqslant\int_{\Omega}|f_{n}|dx
\leqslant\int_{\Omega}|f|dx,
\end{equation*}
and note that
\begin{align*}
\alpha(\frac{p}{p+\alpha-1})^{p}\int_{\Omega}|\nabla
u_{n}^{\frac{p+\alpha-1}{p}}|^{p}dx=\alpha\int_{\Omega}|\nabla
u_{n}|^{p}u_{n}^{\alpha-1}dx,
\end{align*}
i.e.\begin{equation}\|u^{\frac{p+\alpha-1}{p}}_{n}\|_{W^{1,p}_{0}}\leqslant
C\|f\|^{\frac{1}{p}}_{L^{1}(\Omega)}.\end{equation}

In order to prove that $u_{n}$ is bounded in
$W^{1,p}_{loc}(\Omega)$, we may choose $\varphi\in
C^{\infty}_{0}(\Omega),~\Omega'=\{x\in\Omega,~\varphi(x)\neq0\}.$
Multiplying the first identity in Problem $(2.1)$ by
$u_{n}|\varphi(x)|^{p}$ and integrating over $\Omega,$ we get
\begin{equation}
\begin{split}
\int_{\Omega}|\nabla
u_{n}|^{p}|\varphi|^{p}dx&+p\int_{\Omega}u_{n}|\varphi(x)|^{p-2}\varphi(x)|\nabla
u_{n}|^{p-2}\nabla u_{n}\nabla\varphi
dx=\int_{\Omega}\frac{f_{n}u_{n}|\varphi(x)|^{p}}{(u_{n}+\frac{1}{n})^{\alpha}}dx
\\&\leqslant\int_{\Omega}\frac{f}{C_{\Omega'}^{\alpha-1}}|\varphi|^{p}dx
\leqslant\frac{\|\varphi\|^{p}_{L^{\infty}(\Omega)}}{C^{\alpha-1}_{\Omega'}}\int_{\Omega}|f|dx.
\end{split}
\end{equation}
And applying $\textrm{Young's}$ inequality with $\epsilon,$ one has
\begin{equation}
\begin{split}
\Big|p\int_{\Omega}u_{n}|\varphi(x)|^{p-2}\varphi(x)|\nabla
u_{n}|^{p-2}\nabla u_{n}\nabla\varphi dx\Big|&\leqslant
p\int_{\Omega}u_{n}|\varphi(x)|^{p-1}|\nabla
u_{n}|^{p-1}|\nabla\varphi|dx\\
&\leqslant\frac{1}{2}\int_{\Omega}|\nabla
u_{n}|^{p}|\varphi|^{p}dx+\frac{p^{2}}{2}|u_{n}|^{p}|\nabla\varphi|^{p}dx.
\end{split}
\end{equation}
Combining Inequality $(4.2)$ with Inequality $(4.3)$, we have
\begin{equation}
\begin{split}
\int_{\Omega}|\nabla u_{n}|^{p}|\varphi|^{p}dx &\leqslant
\frac{\|\varphi\|^{p}_{L^{\infty}(\Omega)}}{C^{\alpha-1}_{\Omega'}}\int_{\Omega}|f|dx+
\frac{p^{2}}{2}|u_{n}|^{p}|\nabla\varphi|^{p}dx\\
&\leqslant
\frac{\|\varphi\|^{p}_{L^{\infty}(\Omega)}}{C^{\alpha-1}_{\Omega'}}|f|_{L^{1}(\Omega)}+
\frac{p^{2}\|\nabla\varphi\|^{p}_{L^{\infty}(\Omega)}}{2}\int_{\Omega}|u_{n}|^{p}dx.
\end{split}
\end{equation}
Then. by $(4.1)-(4.4)$, we get that $u_{n}$ is bounded in
$W^{1,p}_{Loc}(\Omega)$
\end{proof}
\begin{theorem}
Let $f$ be a nonnegative function in $L^{1}(\Omega)(f\not\equiv0)$.
Then Problem $(1.1)$ has a solution $u$ with
$u^{\frac{p+\alpha-1}{p}}\in W^{1,p}_{0}(\Omega)$ satisfying
\begin{equation}
\int_{\Omega}|\nabla u|^{p-2}\nabla u\nabla\varphi
dx+\int_{\Omega}|u|^{p-2}u\varphi
dx=\int_{\Omega}\frac{f\varphi}{u^{\alpha}}dx,~~\forall \varphi\in
C^{\infty}_{0}(\Omega).
\end{equation}
Moreover, suppose that $f\in L^{m}(m\geqslant1).$ Then the solution
$u$ of Problem $(1.1)$ has the properties
\begin{align*}
&~(i)~\text{If} ~~m>\frac{N}{p}~~and~~2-\frac{2}{m}<p<N, ~then ~u\in L^{\infty}(\Omega). \\
&~(ii)~\text{If}~~1\leqslant m<\frac{N}{p}, ~then ~u\in
L^{s}(\Omega), with~1<s\leqslant\frac{Nm(p+\alpha-1)}{N-pm}.
\end{align*}
\end{theorem}
\begin{proof} $\mathbf{Part 1}$ (Existence).
Following the lines of the proof of Theorem 2.1, we obtain the
existence of the solution of Problem $ (1,1).$

$\mathbf{Part 2}$ (Regularity). (i) Using a trick of Theorem 2.1, it
is easy to see that the first conclusion holds.

 (ii).~Define $\delta=\frac{m(p+\alpha-1)(N-p)}{p(N-m p)}>\frac{p+\alpha-1}{p}>1.$
Choosing $u_{n}^{p(\delta-1)+1}$ as a test function, and applying
$\hbox{H\"{o}lder's}$ inequality,  we arrive at
\begin{equation}
\begin{split}
(p(\delta-1)+1)\int_{\Omega}|\nabla u_{n}|^{p}u_{n}^{p(\delta-1)}dx+
\int_{\Omega}u_{n}^{p\delta}dx&=\int_{\Omega}\frac{f_{n}}{(u_{n}+\frac{1}{n})^{\alpha}}u_{n}^{p(\delta-1)+1}dx
dx\\&\leqslant\int_{\Omega}f(x)u_{n}^{p(\delta-1)+1-\alpha}dx\\
&\leqslant\Big(\int_{\Omega}u_{n}^{(p(\delta-1)+1-\alpha)\frac{m}{m-1}}dx\Big)^{\frac{m-1}{m}}\|f\|_{L^{m}(\Omega)}.
\end{split}
\end{equation}
Moreover, applying $\textrm{Sobolev}$ embedding theorem and using
Inequality $(2.11)$, we get
\begin{equation}
\begin{split}
\Big(\int_{\Omega}|u_{n}|^{\frac{\delta
Np}{N-p}}dx\Big)^{\frac{N-p}{N}}&\leqslant C\int_{\Omega}|\nabla
u_{n}^{\delta}|^{p}dx\leqslant
C\delta^{p}\int_{\Omega}u_{n}^{p(\delta-1)}|\nabla u_{n}|^{p}dx
\\&\leqslant\Big(\int_{\Omega}u_{n}^{p(\delta-1+1-\alpha)\frac{m}{m-1}}dx\Big)^{\frac{m-1}{m}}\|f\|_{L^{m}(\Omega)}.
\end{split}
\end{equation}

Using the definition of $\delta$, we get $\delta
p^{*}=(p(\delta-1)+1-\alpha)\frac{m}{m-1},$ then for all
$~1\leqslant s\leqslant\frac{Nm(p+\alpha-1)}{N-mp}$, we have
\begin{equation*}
\|u_{n}\|_{L^{s}(\Omega)}\leqslant
C\|f\|^{\frac{1}{p+\alpha-1}}_{L^{m}(\Omega)}.
\end{equation*}
\end{proof}

Second, we study the existence and nonexistence of positive solutions in the case $f(x)\in L^{\infty}(\Omega).$ Our main results are as follows
\begin{theorem}If
$1<\alpha<2,~~f(x)\in L^{\infty}(\Omega),$ then Problem $(1.1)$ has a unique positive solution in $W^{1,p}_{0}(\Omega).$
\end{theorem}
\begin{proof}
Multiplying the first identity in Problem $(2.1)$ by $u_{n}$, integrating over $\Omega,$
and applying Lemma 2.2 and Lemma 3.3, we get
\begin{equation*}
\int_{\Omega}|\nabla
u_{n}|^{p}dx+\int_{\Omega}|u_{n}|^{p}dx=\int_{\Omega}\frac{f_{n}u_{n}}
{(u_{n}+\frac{1}{n})^{\alpha}}dx\leqslant\int_{\Omega}|f_{n}|u^{1-\alpha}_{1}dx
\leqslant\|f\|_{L^{\infty}(\Omega)}\|u^{(1-\alpha)}_{1}\|_{L^{1}(\Omega)},
\end{equation*}
ie. \begin{align*} \|u_{n}\|_{W^{1,p}_{0}(\Omega)}\leqslant(\|f\|_{L^{\infty}(\Omega)}\|u^{(1-\alpha)}_{n}\|_{L^{1}(\Omega)})^{\frac{1}(p)}<\infty.\end{align*}
Once $W^{1,p}_{0}(\Omega)$ estimates are obtained,  we may prove the existence of solutions with similar methods of the proof of Theorem 2.1.
\end{proof}
Next, we consider $\alpha>2.$ Similarly as the proof of Lemma 3.3 or according to Lemma 3.2 in
\cite{NHLKS}, we have
\begin{lemma}
 Let $\Phi\in C(\bar{\Omega})$ be
an eigenfunction corresponding the first eigenvalue $\lambda$ of the
following problem
\begin{equation*}
\begin{cases}
-\Delta_{p}\Phi=\lambda \Phi^{p-1},~~&x\in\Omega;\\
\Phi>0,~~&x\in\Omega;\\
\Phi=0,~~&x\in\partial\Omega,
\end{cases}
\end{equation*}
then, $\int_{\Omega}\Phi^{r}dx<\infty\Longleftrightarrow r>-1.$
\end{lemma}
\begin{theorem}
If $\alpha>2,f(x)\in L^{\infty}(\Omega),~\hbox{and}~f(x)\geqslant
\Phi^{\frac{1}{p-1+\alpha}},~in~\overline{\Omega},$ then the
solution of Problem $(1.1)$ is not in $W^{1,p}_{0}(\Omega).$
\end{theorem}
\begin{proof}
With similar method in \cite{NHLKS}, we know that there exist
$b>0~\hbox{and}~\sigma=\frac{p}{p-1+\alpha}$  such that
$0<u(x)\leqslant b\Phi^{\sigma}(x)~in~\Omega.$  Next, we argue by
contradiction. Suppose that the solution $u$ of Problem $(1.1)$
belongs to $W^{1,p}_{0}(\Omega).$ Since $C^{\infty}_{0}(\Omega)$ is
dense in $W^{1,p}_{0}(\Omega)$, we may choose a sequence
$\{Z_{n}\}\subseteqq C^{\infty}_{0}(\Omega)$ satisfying
$$Z_{n}\rightarrow
u,~~in~~W^{1,p}_{0}(\Omega),~~~as~~n\rightarrow\infty.$$ By
\cite{GDMFMLAP,LOAUNN}, we have $Z_{n}^{+}=\max\{Z_{n},0\}\in
W^{1,p}_{0}(\Omega),~n=1,2,3\cdots.$ By means of $u\geqslant 0$, it
is easy to prove that
\begin{align*}&Z_{n}^{+}\rightarrow
u,~~in~~W^{1,p}_{0}(\Omega),~~~as~~n\rightarrow\infty;\\
\end{align*}
Furthermore, According to F.Riesz theorem in\cite{AKJOSF}, we may assume that
$Z_{n}^{+}$ converges to $u$ almost everywhere.   Choosing
$Z_{n}^{+}$ as a test-function in (4.5), we have
\begin{align}
\int_{\Omega}|\nabla u|^{p-2}\nabla u\nabla Z_{n}^{+}dx+
a\int_{\Omega}|u|^{p-2}u
Z_{n}^{+}dx=\int_{\Omega}\frac{f(x)}{u^{\alpha}}Z_{n}^{+}dx.
\end{align}
By the definition of weak convergence, $\hbox{Poincar\'{e}'s}$
inequality and $\hbox{Fatou's}$ lemma, we get
\begin{align}
&\int_{\Omega}|\nabla u|^{p}dx\geqslant \frac{1}{}\int_{\Omega}(|\nabla
u|^{p}+|u|^{p})dx
=\lim\limits_{n\rightarrow\infty}\int_{\Omega}(|\nabla
u|^{p-2}\nabla u\nabla Z_{n}^{+}+ |u|^{p-2}u
Z_{n}^{+})dx;\\
&\int_{\Omega}f(x)u^{1-\alpha}dx\leqslant
\lim\limits_{n\rightarrow\infty}\int_{\Omega}f(x)Z_{n}^{+}u^{-\alpha}dx
\end{align}

Again by $f(x)\geqslant
\Phi^{\frac{1}{p-1+\alpha}},~in~\overline{\Omega}$ and
$0<u(x)\leqslant b\Phi^{\sigma}(x),~in~\Omega,$ we have
\begin{align}
\int_{\Omega}f(x)u^{1-\alpha}dx\geqslant C(M,\alpha,p)
\int_{\Omega}\Phi^{\sigma(1-\alpha)+\frac{\sigma}{p}}dx=+\infty~~\Longleftarrow
\alpha>2.
\end{align}
By $(4.8)-(4.11),$  we  have
$$\int_{\Omega}|\nabla u|^{p}dx=+\infty,$$ which contradicts the
assumption that $u\in W^{1,p}_{0}(\Omega).$
\begin{remark}Due to technical reasons, the authors have to give a technical
condition to $f(x)$. It is not clear whether the problem still has
no a solution  in $W^{1,p}_{0}(\Omega)$ if the restriction that $f(x)\geqslant
\Phi^{\frac{1}{p-1+\alpha}}$ is
removed. Besides, what happens to the solution in the case when $\alpha=2$? whether is $\alpha=2$ a critical exponent or not? Up to now, this is still an open problem.
\end{remark}
\noindent{\bf ACKNOWLEDGMENTS:} The first author thanks the
hospitality of the Department of Mathematics, Michigan State
University, and the first author¡¯s research was partially supported
by the China Scholarship Council, partially supported by NSFC
(11301211) and by Fundamental Research Funds of Jilin
University(450060491473). The second author's research was partially
supported by NSFC (11271154) and by the 985 program of Jilin
University
\end{proof}


\begin{thebibliography}{99}
\bibitem{MOTANI}
M.~\^{O}tani,~Existence and Nonexistence of nontrivial solutions of
some nonlinear degenerate elliptic
equations.~J.~Func.~Anal.~76(1988):¡¤140-159.
\bibitem{GDMFMLAP}
G.~D.~Maso,~F.~Murat,~L.Orsina,~A.~Prignet,~Renormalized solutions
of elliptic equations with general measure
data.~Ann.~Scuola~Norm~Sup.~Pisa Ci.~Sci.~28(1999):~741-808.

\bibitem{LBTG}
L.~Boccardo,~T.~Gallouet,~Nonlinear elliptic equations with rights
hand side measures.\\~Comm.~Partial~Diff.~Equ.~17(364)(1992):~641-655.

\bibitem{LBTG1}
L.~Boccardo,~T.~Gallouet,~Nonlinear elliptic and parabolic equations
involving measure data.~J.~Func.~Anal.~87(1989):~149-169.

\bibitem{ACLPJM}
A.~C.~Lazer,~P.~J.~Mckenna,~On a singular nonlinear elliptic
boundary value problem.~Proc.\\~Amer.~Math.~Soc.,~111(1991):721-730.



\bibitem{ZZJC}
Z.~Zhang,~J.~Cheng,~Existence and optimal estimates of solutions for
singular nonlinear Dirichlet problems.~Nonlinear~Anal.,~57(2004)
:473-484.

\bibitem{YSSWYL}
Y.~Sun,~S. Wu,~Y.~Long,~Combined effects of singular and superlinear
nonlinearities in some singular boundary value problems.~J. Diff.~
Equ.,~176 (2001):511-531.

\bibitem{ALAS}
A.~Lair,~A.~W.~Shaker,~Classical and weak solutions of a singular
semilinear elliptic problem.~J.~Math.~Anal.~Appl.,~211
(1997):371-385.
\bibitem{LBLO}
L.~Boccardo,~L.~Orsina,~Semilinear elliptic equations with singular
nonlinearities.~Calc.~Var.\\~Partial~Diff.~Equ.~37(2010):~363-380.

\bibitem{NHCSNS}
N.~Hirano,~C.~Saccon,~N.~Shioji,~Existence of multiple positive
solutions for singular elliptic problems with concave and convex
nonlinearities.~Adv.~Diff.~Equ.~9(2004):197-220.

\bibitem{MCGP} M.~Coclite ,~G.~Palmieri,~On a singular nonlinear Dirichlet
problem. Comm. Partial Diff.\\~Equ.,~14(1989):1315-1327.

\bibitem{JGJST}
J.~Giacomoni,~J.~Schindler,~Tak\'{a}\v{c},~Sobolev versus H\"{o}lder
local minimizers and existence of\\ multiple solutions for a
singular quasilinear
equation.~Ann.~Scuola~Norm.~Sup.~Pisa~Cl.~Sci.,~6 (2007):117-158.

\bibitem{NHLKS}
N.~Hoang LOC,~K.~Schmitt,~Boundary value problems for singular
elliptic equations.~Rocky\\~Mountain~J.~Math.~41(2)(2011):555-572.

\bibitem{PT}
P.~Tolksdorf,~regularity for a more general class of quasilinear
elliptic equations.~J.~Diff.\\~Equ.~51(1984):~126-150.


\bibitem{FCMGVR}
F. C\^{i}rstea, M. Ghergu,~V. R\v{a}dulescu, Combined effects of
asymptotically linear and singu-\\lar nonlinearities in bifurcation
problems of Lane-Emden-Fowler type. J.~Math.~Pures~Appl.,\\~84
(2005):493-508.



\bibitem{LOAUNN}
O.~A.~Ladyzhenskaya,~N.~N.~Uraltseva,~Linear and quasilinear
elliptic equations. New,York:\\Academic press,~1968.

\bibitem{ADRD}
A.~David,~R.~David,~The ambrosetti-Prodi for the p-Laplace
operator.~Comm.~Partial~Diff.\\~Equ.~31(6)(2006):~849-865.


\bibitem{JLV}
J.~L.~V\'{a}zquez,~A strong maximum principle for some quasilinear
elliptic equations.~Appl.\\~Math.~Optim. 12(1984):191-202.

\bibitem{AKJOSF}
A.~Kufner,~J.~Oldrich,~S.~Fucik,~Function space,~Kluwer Academic
publishers,~1977.









\end{thebibliography}
\end{document}